\documentclass[a4paper]{article}

\usepackage[
	language=english,
	title={On Anick resolution: from the original setting to the language of non-commutative Gröbner bases},
	subject={Homological algebra},
	author={Adya MUSSON-LEYMARIE},
	date={\today},
	font={serif},
	theoremstyle={plain},
]{packages/adyadefault/adyadefault}

\newcommand{\Kbb}{\ensuremath{\mathbb{K}}}
\newcommand{\freeAlgebra}[2][]{\ensuremath{#1\left\langle#2\right\rangle}}
\newcommand{\LM}[1]{\ensuremath{\fP[\mathrm{LM}]{#1}}}
\newcommand{\LC}[1]{\ensuremath{\fP[\mathrm{LC}]{#1}}}

\begin{document}
	\makeatletter
\begin{center}
    \huge \@title \\\vspace*{.5em}
    \Large \@author \\\vspace*{.5em}
    Univ. Limoges, CNRS, XLIM, UMR 7252, F-87000 LIMOGES, France  \\
    \large email: \texttt{adya.musson-leymarie@protonmail.com}
\end{center}
\makeatother

	\begin{abstract}
		Anick introduced a resolution, that now bears his name, of a field using an augmented algebra over that field. We present here what one could call a dictionary between Anick's original paper \cite{Anick-86} and the other resources on the matter, most of which use the language of non-commutative Gröbner bases.
	\end{abstract}

	\section*{Introduction}
\addcontentsline{toc}{section}{Introduction}

When the task at hand is to effectively compute some homology groups of an associative algebra, one is presented with a choice: which resolution to use. The bar resolution~\cite{MacLane-95} is a resolution that always exists but is usually far too large for practical purposes. One prefers resolutions that are as close as possible to the minimal one. However, there are no definite algorithm known to compute the minimal resolution in general. Anick resolution~\cite{Anick-86} offers an alternative (for a certain class of algebras) that involves usable amount of data to be constructed and used to complete our goals of computation. It is not in general minimal but it is still smaller than the bar resolution and is computed quite easily algorithmatically.

Anick resolution acts upon an augmented algebra $A$ with a given presentation whose generators extend in a free monoid of monomials that needs be equipped with a monomial order. The resolution can be taken in the category of right $A$-modules as well as left $A$-modules. It consists of free modules whose bases are formed from \emph{$n$-chains}~\cite{Anick-85, Anick-86}, a construction that acts purely combinatorially on another one called \emph{obstructions} (or \emph{tips}). These two are the fundamental concepts surrounding Anick resolution. An algorithmic method to compute the so-called obstructions is through the use of non-commutative Gröbner bases, as done by Ufnarovski in~\cite{Ufnarovski-95}. This is why the Anick resolution is nowadays mostly known through the lense of non-commutative Gröbner bases.

Our purpose in this paper is to give proofs and explanations from known facts in the folklore surrounding the topic of Anick resolution, that had yet to be written explicitly. In particular, we wish to emphasise the bridge existing between Anick's original paper~\cite{Anick-86} and the subsequent resources expressed in the language of non-commutative Gröbner bases. Following Anick's paper structure, we will introduce our setting while showing how it translates with his work. In the first section, the two main results are Proposition~\ref{proposition:M-oir} and Proposition~\ref{proposition:V_M-basis} that state, respectively:
\begin{itemize}
    \item the \emph{normal words} are exactly the words that have no leading monomials of the relations as subwords,
    \item the \emph{obstructions} are the leading monomials of the minimal non-commutative Gröbner bases of the ideal of relations.
\end{itemize}
The second section introduces the notion of $n$-chains from two different perspectives: one due to Anick and the other one due to Ufnarovski in terms of a graph. We show in Proposition~\ref{proposition:chains-definitions-match} that those two are equivalent. The third section presents the resolution in itself and gives the proof written by Anick, as a matter of completeness, with somewhat more details to help the reader. Throughout this paper, we will follow an example to show how the different constructions concretely shape out to be.

\section*{Notations and conventions}
\addcontentsline{toc}{section}{Notations and conventions}

\begin{itemize}
    \item We will denote by $\freeAlgebra{X}$ (where $X$ is a non-empty set) the \emph{free monoid} on $X$ consisting of the words on the alphabet $X$. We will write $1$ for the empty word. Alternative notations often found in the litterature are $X^\ast$ (\emph{Kleene star}) for the free monoid and $\epsilon$ for the empty word.
    \item Writing $\Kbb S$ where $\Kbb$ is a field and $S$ is a set will denote the $\Kbb$-vector space with basis $S$, thus consisting of the finite formal linear combinations of elements of $S$ with coefficients in $\Kbb$.
    \item We identify $\freeAlgebra[\Kbb]{X}$, where $\Kbb$ is a field and $X$ is a non-empty set of indeterminates, with the \emph{free algebra} consisting of all the polynomials with non-commutative indeterminates in $X$ and coefficients in $\Kbb$. It is the $\Kbb$-algebra on the free monoid $\freeAlgebra{X}$ (see for instance \cite{Brevsar-14} for the construction of the free algebra). Alternative notations in the litterature include $\Kbb X^\ast$ or $T(V)$, where $T(V)$ denotes the \emph{tensor algebra} constructed from a vector space $V$ whose basis is in bijection with $X$. The tensor algebra so-defined and $\freeAlgebra[\Kbb]{X}$ are isomorphic as $\Kbb$-algebras.
    \item The notation $I(R)$, where $R$ is a subset of $\freeAlgebra[\Kbb]{X}$, means the two-sided ideal generated by $R$ in the ring $\freeAlgebra[\Kbb]{X}$, \ie the set of finite combinations of elements from $R$ with left and right coefficients in $\freeAlgebra[\Kbb]{X}$.
    \item Let $A$ be a $\Kbb$-algebra. A \emph{presentation} $\freeAlgebra{X|R}$ of $A$ is given by a set $X$ of indeterminates called \emph{generators} and a subset $R$ of $\freeAlgebra[\Kbb]{X}$ called \emph{relations} such that $A$ is isomorphic to the quotient algebra $\nicefrac{\freeAlgebra[\Kbb]{X}}{I(R)}$. It is easy to see that any algebra is of this form (refer to \cite{Brevsar-14} for details). Given a presentation $\freeAlgebra{X|R}$ of $A$, we will write $\ol{g}$, where $g$ is a polynomial or a word in $\freeAlgebra[\Kbb]{X}$, for the image of $g$ in $A$ by the natural projection $\pi$ induced by the presentation.
    \item Let us fix throughout this paper $\prec$, a monomial order on $\freeAlgebra{X}$, \ie a well-order on $\freeAlgebra{X}$ compatible with left- and right-multiplication of words.
    \item If $f$ is a nonzero polynomial in $\freeAlgebra[\Kbb]{X}$, then the highest monomial for $\prec$ in the support of $f$ (\ie the set of monomials in $f$ appearing with a nonzero coefficient) will be written $\LM{f}$. We will write similarly $\LM{F} := \ens{\LM{f}}{f \in F}$ for any set $F$ of nonzero polynomials in $\freeAlgebra[\Kbb]{X}$.
    \item For a set of generators $X$, we will call \emph{monomial ideal} any monoidal ideal in $\freeAlgebra{X}$, \ie any subset $I$ of $\freeAlgebra{X}$ such that for every word $w$ in $I$ and every word $u$ and $v$ in $\freeAlgebra{X}$, the word $uwv$ is also in $I$.
    \item We remind the reader that, given an ideal $I$ in $\freeAlgebra[\Kbb]{X}$, a \emph{non-commutative Gröbner basis} of $I$ according to the monomial order $\prec$ is any subset $G$ of $I$ such that the leading monomials $\LM{G}$ generates $\LM{I}$ as a monomial ideal, \ie every leading monomial in $I$ has a leading monomial in $G$ as a subword. We will call a non-commutative Gröbner basis $G$ \emph{minimal} if no leading monomial in $\LM{G}$ is divisible by another leading monomial in $\LM{G}$. If moreover no monomial in the support of any element of $G$ is divisible by a leading monomial in $\LM{G}$, it is said to be \emph{reduced}. Note that all the minimal non-commutative Gröbner bases of a same ideal $I$ share the same set of leading monomials.
\end{itemize}

	\section{Setting}

Let us fix a field $\Kbb$ throughout this paper. This corresponds to the field $k$ in \cite{Anick-86}.

Suppose we have an associative unitary $\Kbb$-algebra denoted $A$ (instead of $G$ in \cite{Anick-86}) that we assume is augmented by the augmentation map $\fHV[\eps]{A}{\Kbb}$; it is a surjective homomorphism of $\Kbb$-algebras. We will denote $\fHV[\eta]{\Kbb}{A}$ the section of $\eps$ defined by $\eta(1_\Kbb) = 1_A$, as a $\Kbb$-linear map.

Throughout this paper, we will assume that $A$ is defined by a fixed presentation $\freeAlgebra{X|R}$ in the sense that $A$ is actually equal to the quotient algebra $\nicefrac{\freeAlgebra[\Kbb]{X}}{I(R)}$. In \cite{Anick-86}, Anick uses presentations implicitly: he picks out a set $X$ (he denotes $S$) of generators for $A$, and considers the canonical surjective morphism of algebras $\fHV[f]{\freeAlgebra[\Kbb]{X}}{A}$ that ensues, that subsequently, by the First Isomorphism Theorem, gives rise to an isomorphism $\nicefrac{\freeAlgebra[\Kbb]{X}}{\ker(f)} \cong A$. In our setting, this surjective morphism $f$ is exactly the natural projection $\pi$ from the free algebra to the quotient algebra. It follows that the kernel $\ker(f)$ in \cite{Anick-86} is nothing other than the two-sided ideal $I(R)$ generated by $R$ in our notations.

It is worth noting that most subsequent sources make the assumption that $\eps$ is zero on the set $X$ of generators. This indeed is the case when we take the very common augmentation of a presented algebra as the evaluation of polynomials at $0$, assuming our relations do not contain any constant terms. It is also the case when we consider the common case of a connected graded algebra augmented by its natural augmentation, as done in \cite{Ufnarovski-95}\footnote{Where connectivity of graded algebras is always assumed at page 24.}. However, with very little efforts, mostly during the initialisation of the proof of exactness, the general case, without any assumptions on the augmentation map $\eps$, remains true as shown in \cite{Anick-86}.

In \cite{Anick-86}, Anick uses a specific kind of monomial order that is graded by a certain function he denotes~$e$. In particular, if we follow his suggestion of setting $e(x) = 1$ for all $x \in X$, we obtain the order commonly known as \emph{deglex} (lexicographic order graded by degree).

In \cite[Lemma 1.1]{Anick-86}, Anick introduces a set $M$, defined as the set of words $w$ whose images in $A$ cannot be expressed as a finite linear combination of images in $A$ of words that are smaller than $w$ by $\prec$. These words are called \emph{normal words} in the litterature. This vocabulary comes from the algebraic rewriting area where a word is called a \emph{normal form} according to a set of rewriting rules when no more of the rules can be applied to it. We shall later see (Corollary~\ref{corollary:R-is-grobner-basis}) that the set $M$ is exactly the set of normal form monomials according to $R$ (\ie the rewriting rules induced by the relations) if, and only if, $R$ is a non-commutative Gröbner basis of $I(R)$ according to $\prec$. Therefore, by existence and uniqueness of a reduced non-commutative Gröbner basis of $I(R)$, it makes sense to talk about normal words solely based on the ideal rather than the generating set of rewriting rules.
What Anick calls an \emph{admissible monomial} in \cite{Anick-86} is thus in our language a \emph{normal word}.

In our setting, we will define and write $M$ in the same manner, explicitly:
\[
    M := \ens{w \in \freeAlgebra{X}}{\forall (\DOTS{w_1}{w_n}) \in \freeAlgebra{X}^n \vir w_i \prec w \ALORS \forall (\DOTS{\lambda_1}{\lambda_n}) \in \Kbb^n \vir \ol{w} \neq \sum_{i = 1}^{n} \lambda_i \ol{w_i}}.
\]

It is well known in the litterature on non-commutative Gröbner bases~\cite{Ufnarovski-95} that the set $M$ is the complement of $\LM{I(R)}$ in $\freeAlgebra{X}$, usually denoted $O(I(R))$. We prove this fact in the next proposition.
\begin{proposition}{}{M-oir}
    With the same previous notations, we have:
    \[
        M = \freeAlgebra{X} \setminus \LM{I(R)} =: O(I(R)).
    \]

    As a consequence, we have the very well-known direct sum decomposition of $\Kbb$-vector spaces:
    \[
        \freeAlgebra[\Kbb]{X} = I(R) \oplus \Kbb M.
    \]
\end{proposition}
\begin{proof}
    Let $w \in M$. Suppose there exists a nonzero polynomial $g \in I(R)$ such that $w = \LM{g}$. Therefore, we can write $g = \lambda_w w + \sum_{w' \prec w} \lambda_{w'} w'$ with $\lambda_w \neq 0$. Applying the natural projection $\pi$, on one hand, we get $\ol{g} = 0$ because $g \in I(R)$ and on the other hand, $\ol{g} = \lambda_w \ol{w} + \sum_{w' \prec w} \lambda_{w'} \ol{w'}$ because $\pi$ is linear. By rearranging, we have exhibited that:
    \[
        \ol{w} = -\frac{1}{\lambda_w} \sum_{w' \prec w} \lambda_{w'} \ol{w'}.
    \]
    Therefore, $w \notin M$ which is a contradiction, so $M \subseteq O(I(R))$.

    Conversely, if $w \notin M$, \ie we can write $\ol{w} = \sum_{w' \prec w} \lambda_{w'} \ol{w'}$, then consider the polynomial $g = w - \sum_{w' \prec w} \lambda_{w'} w'$; it is non-zero and is trivially sent to zero under $\pi$, therefore we exhibited $g \in I(R)$ such that $\LM{g} = w$, which means $w \in \LM{I(R)}$, \ie $\freeAlgebra{X} \setminus M \subseteq \LM{I(R)}$ from which we deduce $O(I(R)) \subseteq M$ and the result follows.
\end{proof}

\begin{corollary}{}{R-is-grobner-basis}
    The set $R$ of relations is a non-commutative Gröbner basis of $I(R)$ if and only if $M$ is the set of normal form monomials according to $R$.
\end{corollary}

\begin{proof}
    Normal form monomials according to $R$ are exactly the monomials that are not in the monomial ideal generated by $\LM{R}$. Therefore, the corollary is a consequence of Proposition~\ref{proposition:M-oir}, by the definition of non-commutative Gröbner bases given previously.
\end{proof}

Since $A$ is isomorphic to $\nicefrac{\freeAlgebra[\Kbb]{X}}{I(R)}$, it follows, from the decomposition in Proposition~\ref{proposition:M-oir}, that $A$ is isomorphic to $\Kbb M$ as $\Kbb$-vector spaces, \ie the family $\ol{M} := \famille{\ol{m}}{m \in M}$ is a $\Kbb$-basis of $A$. In particular, the cardinality of the set of normal words gives the dimension of $A$. Some authors (see for instance \cite{Ufnarovski-95}, page 28) write $N$ for the space spanned by normal words $\Kbb M$ and call it the \emph{normal complement} of the ideal $I(R)$: it is isomorphic to $A$ as $\Kbb$-vector spaces and allows therefore, by identification, to perform most of our computations inside the free algebra.

The set $M$ has a special structure that Anick calls \emph{order ideal of monomials} (or "o.i.m." for short). That structure is defined as a subset $W$ of words in $\freeAlgebra{X}$ such that every subword of a word in $W$ is also in $W$. He proceeds by mentionning that giving an o.i.m. is equivalent to giving an anti-chain with respect to the subword partial order (\ie a set of words that are pairwise not subwords of one another). We prove that result here in the more general context of any poset with a well-founded relation (o.i.m.'s and anti-chains are defined for any poset). If $\hat{E} := (E, \leq)$ is a partially ordered set, consider the following sets:
\begin{equation*}
    I_{\hat{E}} := \ens{F \subseteq E}{\forall x \in F, \forall y \in E, y \leq x \ALORS y \in F}, \label{eq:oim}
\end{equation*}
\begin{equation*}
    J_{\hat{E}} := \ens{F \subseteq E}{\forall x \in F, \forall y \in F, x \not< y \ET y \not< x}. \label{eq:anti-chain}
\end{equation*}

Notice that the set $I_{\hat{E}}$ is the set of o.i.m.'s of $\hat{E}$ and the set $J_{\hat{E}}$ is the set of anti-chains of $\hat{E}$.

Then, define the following map:
\begin{equation*}
    \fHV[f_{\hat{E}}]{J_{\hat{E}}}{I_{\hat{E}}}
    [F][f_{\hat{E}}(F) := \ens{y \in E}{\forall x \in F \vir (x \leq y \OU y \leq x) \ALORS y < x}.] \label{eq:fhatE}
\end{equation*}

This translates as saying that an element $y$ is in the image of an anti-chain $F$ if and only if, granted $y$ is comparable with an element from $F$, then it is necessarily smaller. This means that $f_{\hat{E}}(F)$ is exactly the union of the set of the elements incomparable with any element from $F$ and of the set of the elements that are smaller than an element from $F$.

The map is well-defined because if $F$ is an anti-chain, then, for any $y \in f_{\hat{E}}(F)$ and $x \leq y$:
\begin{itemize}
    \item if $y$ is comparable to an $x' \in F$, then $y < x'$. By transitivity, $x < x'$ and thus $x \in f_{\hat{E}}(F)$.
    \item if $y$ is incomparable with every element of $F$, then $\forall x' \in F, x' \not\leq x$. Otherwise there would exist $x' \in F$ such that $x' \leq x \leq y$, a contradiction. Therefore, if $x$ is comparable with a $x' \in F$, it is necessarily such that $x < x'$. This means exactly that $x \in f_{\hat{E}}(F)$.
\end{itemize}
Hence, $f_{\hat{E}}(F)$ is an o.i.m.

\begin{proposition}{}{oim-antichain}
    Let $\hat{E} = (E, \leq)$ be a partially ordered set. If $\leq$ is a well-founded relation on $E$, then $f_{\hat{E}}$ is a bijection and its inverse is given by:
    \begin{equation}
        \fHV[g_{\hat{E}}]{I_{\hat{E}}}{J_{\hat{E}}}
        [F'][g_{\hat{E}}(F') := \ens{y \in E \setminus F'}{\forall x \in E \setminus F' \vir x \not< y}.] \label{eq:fhatEinverse}
    \end{equation}
\end{proposition}

Notice that the map $g_{\hat{E}}$ sends an o.i.m. $F$ to the set of minimal elements in $E \setminus F$, which is an anti-chain by construction; hence $g_{\hat{E}}$ is well-defined. Since $\leq$ is a well-founded relation, note that $g_{\hat{E}}(F)$ is non-empty if and only if $F \neq E$.

\begin{proof}
    Let us show both assertions of the proposition at once by proving that $g_{\hat{E}}$ is indeed a two-sided inverse for $f_{\hat{E}}$.

    Consider $F' \in I_{\hat{E}}$ an o.i.m. Define $F$ as $g_{\hat{E}}(F')$, \ie as the set of minimal elements of $E \setminus F'$.
    
    If $F' = E$, then $F = \emptyset$. In that case, $f_{\hat{E}}(F)$ is equal to $E$ since there are no elements in $F$ to compare the elements of $E$ with.
    
    Consider thus $F' \neq E$. Hence, $F$ is non-empty. It follows that see that $f_{\hat{E}}(F) = F'$. Indeed:
    \begin{itemize}
        \item Suppose $y \in f_{\hat{E}}(F)$. On one hand, if $y$ is not comparable with any elements of $F$, then $y$ cannot possibly be in $E \setminus F'$ since there exists elements in $F$ and they are minimal in $E \setminus F'$; $y$ would therefore be comparable with one of them. On the other hand, if it is comparable to an $x \in F$, then $y < x$ but $y$ cannot be in $E \setminus F'$ since $x$ is minimal. Therefore, $f_{\hat{E}}(F) \subseteq F'$.
        \item Suppose now $y \in F'$. Assume $y$ is comparable with some $x \in F$ \ie $x \leq y \lor y \leq x$. If $x \leq y$ then it would follow that $x \in F'$ since $F'$ is an o.i.m. Therefore, since $x \notin F'$, we must necessarily have $y < x$ and thus $y \in f_{\hat{E}}(F)$. Hence, $F' \subseteq f_{\hat{E}}(F)$.
    \end{itemize}

    Hence, $f_{\hat{E}} \circ g_{\hat{E}} = \id_{I_{\hat{E}}}$.

    Consider now $F \in J_{\hat{E}}$ an anti-chain. Define $F'$ as $f_{\hat{E}}(F)$. Note that $F \subseteq E \setminus F'$.

    \begin{itemize}
        \item Suppose $y \in g_{\hat{E}}(F')$. In particular, $y \notin F'$. Then $y$ is comparable with an $x \in F$ such that necessarily $x \leq y$. But, on one hand, $x \in F$ implies $x \notin F'$, on the other hand, $y$ is minimal in $E \setminus F'$, therefore $x = y$ and thus $y \in F$.
        \item Suppose $y \in F$. Then $y \notin F'$. Suppose we would have $x \in E \setminus F'$ such that $x < y$. Then $x$ would be comparable with a $z \in F$ such that $z \leq x$ and thus we would have $z < y$. However, $F$ is anti-chain and $z, y \in F$, so that's a contradiction. We conclude that there are no elements $x \in E\setminus F'$ such that $x < y$, which exactly means $y \in g_{\hat{E}}(F')$.
    \end{itemize}
    Hence, we have $g_{\hat{E}} \circ f_{\hat{E}} = \id_{J_{\hat{E}}}$.
\end{proof}

\begin{definition}{}{V_M-anti-chain}
    Denote by $\hat{X}$ the free monoid $\freeAlgebra{X}$ equipped with the well-founded relation of subwords. Define $V_M$ as $g_{\hat{X}}(M)$, the unique anti-chain in $\hat{X}$ associated to the o.i.m.\ $M$ of normal words where $g_{\hat{X}}$ is the map defined in (\ref{eq:fhatEinverse}). The elements of $V_M$ are called the \emph{obstructions} (or \emph{tips}) of the o.i.m.\ $M$.
\end{definition}

It is well known in the litterature~\cite{Ufnarovski-95} that $V_M$ is the minimal generating set of $\LM{I(R)}$ as a monomial ideal. We prove this fact in the next proposition. It shows furthermore the connection between Anick's original setting and the language of non-commutative Gröbner bases.
\begin{proposition}{}{V_M-basis}
    The set $V_M$ is the unique minimal set generating $\LM{I(R)}$ as a monomial ideal. In particular, for any minimal non-commutative Gröbner basis $G$ of $I(R)$, we have:
    \[
        V_M = \LM{G}.
    \]
\end{proposition}
\begin{proof}
    Recall by Definition~\ref{definition:V_M-anti-chain}, $V_M$ is the minimal elements of the complement of $M$. But, by Proposition~\ref{proposition:M-oir}, we have $M = \freeAlgebra{X} \setminus \LM{I(R)}$. Therefore, $V_M$ is exactly the set of minimal elements (in terms of the subword relation) of $\LM{I(R)}$, which is exactly equivalent to saying that $V_M$ generates $\LM{I(R)}$ as a monomial ideal. Moreover, $V_M$ being an anti-chain, removing any element from $V_M$ implies losing the ability to generate $\LM{I(R)}$, hence $V_M$ is minimal as a generating set.
\end{proof}

In particular, if $R$ is indeed a minimal non-commutative Gröbner basis as it usually is, then we have $V_M = \LM{R}$. In general, we can use the set of obstructions as a one-to-one index set for the reduced non-commutative Gröbner basis of $I(R)$. It can be also useful in certain contexts to consider the associated monomial algebra presented by $\freeAlgebra{X | V_M}$, for instance to compute more easily the Hilbert series (see~\cite{Cojocaru-97, Draxler-99, Ufnarovski-95}).

The Proposition~\ref{proposition:V_M-basis} is equivalent to, but expressed in a different way than, the Lemma 1.2 from~\cite{Anick-86} stating that every non-normal word contains an obstruction.

	\section{$n$-chains and critical branchings}

The idea of Anick resolution is to construct free $A$-modules with as close to the minimal amount of generators as possible that still allow us to define differentials in a way that gives rise to a resolution with an explicit contracting homotopy. We will consider here the case of right modules (refer to \cite{Nguyen-19} for an adaptation to left modules).

In order to do so, Anick introduces the notions of $n$-prechains and $n$-chains through a top-down definition.
\begin{definition}{Chains (top-down)}{chains-top-down}
    Let $w = x_1 \cdots x_\ell$ be a word in $\freeAlgebra{X}$. Let $n \in \NnonNul$.

    We say that $w$ is a \emph{$n$-prechain} if there exists two $n$-tuples $(\DOTS{a_1}{a_n})$ et $(\DOTS{b_1}{b_n})$ of integers such that:
    \[
        1 = a_1 < a_2 \leq b_1 < a_3 \leq b_2 < a_4 \leq b_3 < \cdots < a_n \leq b_{n-1} < b_n = \ell
    \]
    and
    \[
        \forall i \in \integers{1}{n} \vir x_{a_i} x_{a_i + 1} \cdots x_{b_i - 1} x_{b_i} \in V_M.
    \]

    A $n$-prechain is called a \emph{$n$-chain} if:
    \[
        \forall m \in \integers{1}{n} \vir \forall i \in \integers{1}{b_m - 1} \vir x_1 x_2 \cdots x_i \text{ is not a $m$-prechain.}
    \]
\end{definition}

Intuitively, a $n$-prechain is a sequence of $n$ obstructions where two obstructions in a row overlap each other by at least a character while obstructions separated by at least one obstruction in the sequence do not overlap. A $n$-chain is a $n$-prechain such that the consecutive overlaps are "maximal" in the sense that no other overlap with the same obstructions could have been longer while still satisfying the condition that the obstructions one apart do not overlap. All of the obstructions need not appear in each prechain and chain and the same obstruction can appear several times within a single prechain or chain. 

Notice that the set of $1$-chains according to this definition is exactly the set of obstructions $V_M$.

\begin{example}{}{chains}
    On the alphabet $X = \ensemble{x, y, z}$ with the anti-chain $V_M = \ensemble{xxx, xxyx, yxz}$, we have  $\underline{x\ol{xx}}\ol{x}$ is a $2$-chain, $\underline{xx\ol{x}}\ol{xx}$ is a $2$-prechain but not $2$-chain nor a $3$-prechain. Similarly, $\underline{x\ol{xx}}\ol{yx}$ is a $2$-chain but $\underline{xx\ol{x}}\ol{xyx}$ is a $2$-prechain but not a $2$-chain, since $\underline{x\ol{xx}}\ol{x}$ is a $2$-prechain contained in it and shorter but with the same number of obstructions. It is also not a $3$-prechain because the first obstruction $xxx$ would overlap with the third obstruction $xxyx$. A $3$-chain is for instance $\underline{xxy\ol{x}}\ol{x}\underline{\ol{yx}z}$.
\end{example}

By convention, let the set of $(-1)$-chains be exactly $\ensemble{1}$ and the set of $0$-chains be exactly $X$.

Anick establishes a result in \cite{Anick-86} in the form of Lemma 1.3 stating that for an $n$-chain, the $n$-tuples $(\DOTS{a_1}{a_n})$ and $(\DOTS{b_1}{b_n})$ are uniquely determined. In particular, this means that:
\begin{proposition}{}{split-chains}
    For any $n \in \NnonNul$, any $n$-chain $w = x_1 \cdots x_\ell$ (defined with $(\DOTS{a_1}{a_n})$ and $(\DOTS{b_1}{b_n})$) can be uniquely expressed as $w = vu$, where $v = x_1 \cdots x_{b_{n-1}}$ is an $(n-1)$-chain and $u = x_{b_{n-1} + 1} \cdots x_\ell$ is a normal word.
\end{proposition}

\begin{example}{}{}
    Following the examples given in Example~\ref{example:chains}:
    \begin{itemize}
        \item for the $2$-chain $w = xxxyx$, $v = xxx$, $u = yx$.
        \item for the $2$-chain $w = xxxx$, $v = xxx$, $u = x$.
        \item for the $3$-chain $w = xxyxxyxz$, $v = xxyxxyx$, $u = z$.
    \end{itemize}
\end{example}

The top-down Definition~\ref{definition:chains-top-down} is not particularly easy to grasp, as such, conceptually and even less algorithmically. We will prefer the bottom-up definition given in most other sources and present it here.

First, let us warn that we will be using the numbering proposed in \cite{Nguyen-19} rather the one proposed in \cite{Anick-86}: what Anick calls $0$-chains in the top-down Definition~\ref{definition:chains-top-down} will be $1$-chains for us, $1$-chains will be $2$-chains, and so on. That way, the numbering will match the homology degrees conveniently.

\begin{definition}{Chains (bottom-up) \emph{due to Ufnarovski} \cite{Ufnarovski-95}}{chains-bottom-up}
    With previous notations and remarks, construct a simple directed graph $Q$ whose nodes are:
    \[
        Q_0 = \ensemble{1} \cup X \cup \ens{s \in \freeAlgebra{X}}{s \text{ is a proper suffix of an obstruction}}.
    \]

    The directed edges are defined as follows:
    \begin{align*}
        Q_1 = \ens{(1, x)}{x \in X} \cup \ens{(s, t) \in (Q_0 \setminus \ensemble{1})^2}{st \text{ contains only one obstruction and it is a suffix}}.
    \end{align*}

    For any non-negative integer $n \in \N$, we define the set of $n$-chains as:
    \[
        C_n := \ens{\prod_{i = 0}^{n} w_i}{(1 = w_0, w_1, \cdots, w_n) \text{ are nodes in a path of length $n$ in $Q$ starting at $1$}}.
    \]
\end{definition}

In other words, an $n$-chain is the product of nodes travelling through a path of length $n$ starting at the node $1$. Note that the nodes that are not in the connected component of the node $1$ have no use for our purpose and can therefore be omitted. Note also that we have $C_0 = \ensemble{1}$, $C_1 = X$, and $C_2 = V_M$.

This definition can be rephrased with ease in terms of a recursive definition of $n$-chains with tails as done in~\cite{Malbos-19}.

\begin{proposition}{Top-down and bottom-up definitions match}{chains-definitions-match}
    Let us denote by $\hat{C}_n$ the set of $n$-chains defined in Definition~\ref{definition:chains-top-down} for $n \geq -1$. We have:
    \[
        \forall n \in \N \vir \hat{C}_{n-1} = C_n.
    \]
\end{proposition}
\begin{proof}
    We see easily that this is true for $n \in \ensemble{0, 1, 2}$.

    By induction, suppose this is true for a certain $n \geq 2$.
    
    Let $w \in \hat{C}_n$ defined by $w = x_1 \cdots x_\ell$ and the $n$-tuples $(\DOTS{a_1}{a_n})$ and $(\DOTS{b_1}{b_n})$ for Definition~\ref{definition:chains-top-down}. By Proposition~\ref{proposition:split-chains}, we have: $v := x_1 \cdots x_{b_{n-1}} \in \hat{C}_{n-1}$ and $u := x_{b_{n-1} + 1} \cdots x_\ell \in M$. Since $u$ is a proper suffix of $s = x_{a_n} \cdots x_{b_n}$ (an obstruction), then $u \in Q_0$. By induction hypothesis, $v \in C_n$. Moreover, the last node in the path for $v$ is $t := x_{b_{n-2} + 1}\cdots x_{b_{n-1}}$. Since $w$ is a $n$-prechain, it follows that $b_{n-2} < a_n$, thus $tu$ evidently contains $s$ as a suffix. Furthermore, the last axiom of $n$-chains (top-down) ensures that $s$ is the only obstruction that $tu$ contains. This means exactly that there is an edge between $t$ and $u$ such that the path $v \in C_n$ can be extended with $u$ and gives $w = vu \in C_{n+1}$.

    Let $w \in C_{n+1}$. It is thus defined as a path of $n+1$ length. Let $u$ be the last node of that path and $t$ the node before that. Denote by $v$ the path of length $n$ when we omit $u$. We have $v \in C_{n}$. By induction hypothesis, it follows that $v \in \hat{C}_{n-1}$. Let $(\DOTS{a_1}{a_{n-1}})$ and $(\DOTS{b_1}{b_{n-1}})$ be the $(n-1)$-tuples defining $v$ in Definition~\ref{definition:chains-top-down}. Denote by $s$ the obstruction linking $t$ and $u$. We know that $tu$ contains $s$ as a suffix and that $\ell(s) > \ell(u)$\footnote{$\ell(w')$ is the length of the word $w' \in \freeAlgebra{X}$, \ie the number of letters from $X$ that constitutes the word.} because $u$ is either a letter or a proper suffix of an obstruction. Define $a_{n} := b_{n-1} + \ell(u) - \ell(s) + 1 \leq b_{n-1}$ and $b_{n} := b_{n-1} + \ell(u) > b_{n-1}$. Then the tuples $(\DOTS{a_1}{a_{n}})$ and $(\DOTS{b_1}{b_{n}})$ make $w = vu$ into a $n$-chain (top-down) since $x_{a_{n}} \cdots x_{b_{n}} = s \in V_M$ and no other obstructions is contained in $x_{b_{n-2} + 1} \cdots x_{b_{n}}$. Therefore, $w \in \hat{C}_{n}$.
\end{proof}


\begin{example}{}{}
    Consider again the example $X = \ensemble{x, y, z}$ and $V_M = \ensemble{xxx, xxyx, yxz}$. We have:
    \[
        Q_0 = \ensemble{1, x, y, z, xx, xyx, yx, xz}.
    \]

    The graph is then given by Figure~\ref{fig:chains-graph}. Each arrow that does not start from $1$ is to be understood as indexed by an obstruction, the obstruction satisfying the condition for the directed edge.

    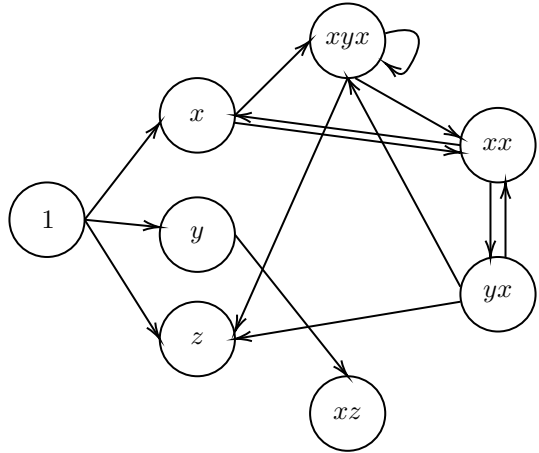
\begin{figure}[h]
        \tikzset{every picture/.style={line width=0.75pt}} 
        \centering
        \caption{$n$-chains graph for Example~\ref{example:chains}}
        \label{fig:chains-graph}
        \begin{tikzpicture}[x=0.75pt,y=0.75pt,yscale=-.75,xscale=.75]

        \draw   (250,275) .. controls (250,261.19) and (261.19,250) .. (275,250) .. controls (288.81,250) and (300,261.19) .. (300,275) .. controls (300,288.81) and (288.81,300) .. (275,300) .. controls (261.19,300) and (250,288.81) .. (250,275) -- cycle ;
        \draw   (150,75) .. controls (150,61.19) and (161.19,50) .. (175,50) .. controls (188.81,50) and (200,61.19) .. (200,75) .. controls (200,88.81) and (188.81,100) .. (175,100) .. controls (161.19,100) and (150,88.81) .. (150,75) -- cycle ;
        \draw   (350,195) .. controls (350,181.19) and (361.19,170) .. (375,170) .. controls (388.81,170) and (400,181.19) .. (400,195) .. controls (400,208.81) and (388.81,220) .. (375,220) .. controls (361.19,220) and (350,208.81) .. (350,195) -- cycle ;
        \draw   (250,25) .. controls (250,11.19) and (261.19,0) .. (275,0) .. controls (288.81,0) and (300,11.19) .. (300,25) .. controls (300,38.81) and (288.81,50) .. (275,50) .. controls (261.19,50) and (250,38.81) .. (250,25) -- cycle ;
        \draw   (150,225) .. controls (150,211.19) and (161.19,200) .. (175,200) .. controls (188.81,200) and (200,211.19) .. (200,225) .. controls (200,238.81) and (188.81,250) .. (175,250) .. controls (161.19,250) and (150,238.81) .. (150,225) -- cycle ;
        \draw   (350,95) .. controls (350,81.19) and (361.19,70) .. (375,70) .. controls (388.81,70) and (400,81.19) .. (400,95) .. controls (400,108.81) and (388.81,120) .. (375,120) .. controls (361.19,120) and (350,108.81) .. (350,95) -- cycle ;
        \draw    (350,95) -- (201.98,75.26) ;
        \draw [shift={(200,75)}, rotate = 7.59] [color={rgb, 255:red, 0; green, 0; blue, 0 }  ][line width=0.75]    (10.93,-3.29) .. controls (6.95,-1.4) and (3.31,-0.3) .. (0,0) .. controls (3.31,0.3) and (6.95,1.4) .. (10.93,3.29)   ;
        \draw    (370,120) -- (370,168) ;
        \draw [shift={(370,170)}, rotate = 270] [color={rgb, 255:red, 0; green, 0; blue, 0 }  ][line width=0.75]    (10.93,-3.29) .. controls (6.95,-1.4) and (3.31,-0.3) .. (0,0) .. controls (3.31,0.3) and (6.95,1.4) .. (10.93,3.29)   ;
        \draw    (280,50) -- (348.26,89.01) ;
        \draw [shift={(350,90)}, rotate = 209.74] [color={rgb, 255:red, 0; green, 0; blue, 0 }  ][line width=0.75]    (10.93,-3.29) .. controls (6.95,-1.4) and (3.31,-0.3) .. (0,0) .. controls (3.31,0.3) and (6.95,1.4) .. (10.93,3.29)   ;
        \draw    (275,50) -- (200.81,218.17) ;
        \draw [shift={(200,220)}, rotate = 293.81] [color={rgb, 255:red, 0; green, 0; blue, 0 }  ][line width=0.75]    (10.93,-3.29) .. controls (6.95,-1.4) and (3.31,-0.3) .. (0,0) .. controls (3.31,0.3) and (6.95,1.4) .. (10.93,3.29)   ;
        \draw    (200,80) -- (348.02,99.74) ;
        \draw [shift={(350,100)}, rotate = 187.59] [color={rgb, 255:red, 0; green, 0; blue, 0 }  ][line width=0.75]    (10.93,-3.29) .. controls (6.95,-1.4) and (3.31,-0.3) .. (0,0) .. controls (3.31,0.3) and (6.95,1.4) .. (10.93,3.29)   ;
        \draw    (200,75) -- (248.59,26.41) ;
        \draw [shift={(250,25)}, rotate = 135] [color={rgb, 255:red, 0; green, 0; blue, 0 }  ][line width=0.75]    (10.93,-3.29) .. controls (6.95,-1.4) and (3.31,-0.3) .. (0,0) .. controls (3.31,0.3) and (6.95,1.4) .. (10.93,3.29)   ;
        \draw    (300,20) .. controls (334.7,10.14) and (320,36.5) .. (314.9,44.07) .. controls (310.29,50.92) and (306.32,45.41) .. (301.54,41.25) ;
        \draw [shift={(300,40)}, rotate = 36.87] [color={rgb, 255:red, 0; green, 0; blue, 0 }  ][line width=0.75]    (10.93,-3.29) .. controls (6.95,-1.4) and (3.31,-0.3) .. (0,0) .. controls (3.31,0.3) and (6.95,1.4) .. (10.93,3.29)   ;
        \draw    (350,190) -- (275.94,51.76) ;
        \draw [shift={(275,50)}, rotate = 61.82] [color={rgb, 255:red, 0; green, 0; blue, 0 }  ][line width=0.75]    (10.93,-3.29) .. controls (6.95,-1.4) and (3.31,-0.3) .. (0,0) .. controls (3.31,0.3) and (6.95,1.4) .. (10.93,3.29)   ;
        \draw    (350,200) -- (201.97,224.67) ;
        \draw [shift={(200,225)}, rotate = 350.54] [color={rgb, 255:red, 0; green, 0; blue, 0 }  ][line width=0.75]    (10.93,-3.29) .. controls (6.95,-1.4) and (3.31,-0.3) .. (0,0) .. controls (3.31,0.3) and (6.95,1.4) .. (10.93,3.29)   ;
        \draw   (50,145) .. controls (50,131.19) and (61.19,120) .. (75,120) .. controls (88.81,120) and (100,131.19) .. (100,145) .. controls (100,158.81) and (88.81,170) .. (75,170) .. controls (61.19,170) and (50,158.81) .. (50,145) -- cycle ;
        \draw    (380,170) -- (380,122) ;
        \draw [shift={(380,120)}, rotate = 90] [color={rgb, 255:red, 0; green, 0; blue, 0 }  ][line width=0.75]    (10.93,-3.29) .. controls (6.95,-1.4) and (3.31,-0.3) .. (0,0) .. controls (3.31,0.3) and (6.95,1.4) .. (10.93,3.29)   ;
        \draw   (150,155) .. controls (150,141.19) and (161.19,130) .. (175,130) .. controls (188.81,130) and (200,141.19) .. (200,155) .. controls (200,168.81) and (188.81,180) .. (175,180) .. controls (161.19,180) and (150,168.81) .. (150,155) -- cycle ;
        \draw    (100,145) -- (148.78,81.59) ;
        \draw [shift={(150,80)}, rotate = 127.57] [color={rgb, 255:red, 0; green, 0; blue, 0 }  ][line width=0.75]    (10.93,-3.29) .. controls (6.95,-1.4) and (3.31,-0.3) .. (0,0) .. controls (3.31,0.3) and (6.95,1.4) .. (10.93,3.29)   ;
        \draw    (100,145) -- (148.01,149.8) ;
        \draw [shift={(150,150)}, rotate = 185.71] [color={rgb, 255:red, 0; green, 0; blue, 0 }  ][line width=0.75]    (10.93,-3.29) .. controls (6.95,-1.4) and (3.31,-0.3) .. (0,0) .. controls (3.31,0.3) and (6.95,1.4) .. (10.93,3.29)   ;
        \draw    (100,145) -- (148.94,223.3) ;
        \draw [shift={(150,225)}, rotate = 237.99] [color={rgb, 255:red, 0; green, 0; blue, 0 }  ][line width=0.75]    (10.93,-3.29) .. controls (6.95,-1.4) and (3.31,-0.3) .. (0,0) .. controls (3.31,0.3) and (6.95,1.4) .. (10.93,3.29)   ;
        \draw    (200,155) -- (273.76,248.43) ;
        \draw [shift={(275,250)}, rotate = 231.71] [color={rgb, 255:red, 0; green, 0; blue, 0 }  ][line width=0.75]    (10.93,-3.29) .. controls (6.95,-1.4) and (3.31,-0.3) .. (0,0) .. controls (3.31,0.3) and (6.95,1.4) .. (10.93,3.29)   ;

        \draw (375,95) node   [align=left] {\begin{minipage}[lt]{68pt}\setlength\topsep{0pt}
        \begin{center}
        $xx$
        \end{center}

        \end{minipage}};
        \draw (175,75) node   [align=left] {\begin{minipage}[lt]{68pt}\setlength\topsep{0pt}
        \begin{center}
        $x$
        \end{center}

        \end{minipage}};
        \draw (275,25) node   [align=left] {\begin{minipage}[lt]{68pt}\setlength\topsep{0pt}
        \begin{center}
        $xyx$
        \end{center}

        \end{minipage}};
        \draw (375,195) node   [align=left] {\begin{minipage}[lt]{68pt}\setlength\topsep{0pt}
        \begin{center}
        $yx$
        \end{center}

        \end{minipage}};
        \draw (275,275) node   [align=left] {\begin{minipage}[lt]{68pt}\setlength\topsep{0pt}
        \begin{center}
        $xz$
        \end{center}

        \end{minipage}};
        \draw (175,225) node   [align=left] {\begin{minipage}[lt]{68pt}\setlength\topsep{0pt}
        \begin{center}
        $z$
        \end{center}

        \end{minipage}};
        \draw (75,145) node   [align=left] {\begin{minipage}[lt]{68pt}\setlength\topsep{0pt}
        \begin{center}
        $1$
        \end{center}

        \end{minipage}};
        \draw (175,155) node   [align=left] {\begin{minipage}[lt]{68pt}\setlength\topsep{0pt}
        \begin{center}
        $y$
        \end{center}

        \end{minipage}};

        \end{tikzpicture}
    \end{figure}

\end{example}

Let us now introduce some useful notations.
\begin{definition}{}{bracket-notation}
    Let $n \in \N$, $m \in \integers{0}{n}$ and $c^{(n)} \in C_n$ be a $n$-chain according to Definition~\ref{definition:chains-bottom-up}. Let us explicitly fix $(\DOTS{a_1}{a_{n-1}})$ and $(\DOTS{b_1}{b_{n-1}})$ the uniquely determined tuples of integers defining $c^{(n)} = x_1 \cdots x_\ell$ as in Definition~\ref{definition:chains-top-down}. Write:
    \[
        \fB{c^{(n)}}^m := \begin{cases}
            1 & \text{if $m = 0$} \\
            x_1 & \text{if $m = 1$} \\
            x_1 \cdots x_{b_{m - 1}} & \text{if $1 < m \leq n$} \\
        \end{cases} \in C_m.
    \]
    This designates the $m$-chain that is a prefix of $c^{(n)}$.
    \[
        \fB{c^{(n)}}_m := \begin{cases}
            w & \text{if $m = 0$} \\
            x_2 \cdots x_\ell & \text{if $m = 1$} \\
            x_{b_{m - 1} + 1} \cdots x_\ell & \text{if $1 < m < n$} \\
            1 & \text{if $m = n$}
        \end{cases} \in \freeAlgebra{X}.
    \]
    That corresponds to the left-over part of the $n$-chain $c^{(n)}$ after removing the $m$-chain prefix.
\end{definition}

In particular, this means that for all $n \in \N$ and $c^{(n)} \in C_n$ then:
\[
    \forall m \in \integers{0}{n} \vir c^{(n)} = \fB{c^{(n)}}^m \fB{c^{(n)}}_m \in \freeAlgebra{X}.
\]

As a remark, notice the link with algebraic rewriting. Let us use some terminology from that field. The relations from $R$ define \emph{rewriting rules}:
\[
    \lambda w_1 \LM{g} w_2 + h \quad\underset{R}{\rightarrow}\quad \frac{\lambda}{\LC{g}} w_1 r(g) w_2 + h
\]
where $w_1, w_2 \in \freeAlgebra{X}$, $\lambda \in \Kbb \setminus \ensemble{0}$, $g \in R$, $h \in \freeAlgebra[\Kbb]{X}$ such that $w_1\LM{g}w_2$ does not belong to its support and $r(g) := \LC{g} \LM{g} - g$ with $\LC{g}$ the coefficient of $\LM{g}$ in $g$.

In that context, a word is said to give rise to a \emph{critical pair} if two (or the same one twice) of those rewriting rules can be applied on parts of the word that overlap, while overall going from beginning to end of the word, giving possibly different results. If three rules can be applied, we talk about \emph{critical triples}, if four, \emph{critical quadruples} and so on. In general, these are called \emph{critical branchings}.

In the common case where $R$ is a minimal non-commutative Gröbner basis (and thus $V_M =~\LM{R}$), let us now note that the set $C_n$ of $n$-chains for $n \geq 3$ is a subset of the words that give rise to a critical branching, \eg $3$-chains give rise to critical pairs. Indeed, a rewriting rule is applied on a word if it contains a leading monomial of $R$ \ie an obstruction in our case. Therefore, two rules will be simulaneoulsy applied on a $3$-chain because it contains exactly two obstructions, and so on and so forth for higher degrees.

For more details on algebraic rewriting and its connections with non-commutative Gröbner bases and Anick resolution, see \cite{Malbos-19}.

	\section{Anick resolution}

We can now introduce the Anick resolution. It will be a resolution of the field $\Kbb$ made out of free right $A$-modules. Indeed, once $A$ is augmented by $\fHV[\eps]{A}{\Kbb}$, we can equip $\Kbb$ with a structure of right $A$-module using the external law of composition $\Kbb \times A \ni (\lambda, a) \mapsto \lambda \eps(a) \in \Kbb$. Similarly, we could define a structure of left $A$-module (or even of $A$-bimodule). Hence, even if in this paper we present a resolution of $\Kbb$ by right $A$-modules, with a few minor adaptations, one by left $A$-modules would work just as well (see \cite{Nguyen-19}).

The free modules in the resolution are defined from the linear hull of (\ie the vector space generated by) the sets of chains in each degree. The differentials are defined inductively at the same time as the contracting homotopy proving that the complex is a resolution.

It is helpful to define an order on the bases of the free modules. In order to do so, we will make use of the monomial order $\prec$ at hand.
\begin{definition}{}{order-basis}
    Let $n \in \N$. Let $C_n$ be the set of $n$-chains on an anti-chain $V$, as defined in Definition~\ref{definition:chains-bottom-up}. We define an order $<$ on the basis of the free right $A$-module $\Kbb C_n \otimes_{\Kbb} A$ as:
    \[
        \forall c_1, c_2 \in C_n \vir \forall s_1, s_2 \in O(I(R)) \vir c_1 \otimes \ol{s_1} < c_2 \otimes \ol{s_2} \overset{\mathrm{def}}{\SSI} c_1s_1 \prec c_2s_2.
    \]
\end{definition}

The order $<$ is well-defined and total because $V$ is an anti-chain (since it entails that $c_1 \otimes \ol{s_1} \neq c_2 \otimes \ol{s_2}$ implies that $c_1 s_1 \neq c_2 s_2$). Moreover, it is a well-order (induced by the properties of $\prec$).


It follows that, for every element in $\Kbb C_n \otimes_{\Kbb} A$, there is a greatest term according to the order $<$, since such an element is a finite sum of terms. An alternative way to define this greatest term would be to use the monomial order on a polynomial that we associate to the element of $\Kbb C_n \otimes_{\Kbb} A$, as we do in the next definition.

\begin{definition}{}{LM}
    The \emph{leading monomial} (or \emph{high-term}, as called by Anick \cite{Anick-86}) of an element $P := \sum_{i} \lambda_i c^{(n)}_i \otimes \ol{r_i} \in \Kbb C_n \otimes_{\Kbb} A$ is defined as
    \[
        \LM{P} := \LM{\sum_{i}\lambda_i\fP{c^{(n)}_i\widehat{r_i}}} \in \freeAlgebra{X},
    \]
    where $\widehat{r_i}$ is the unique normal form of $r_i$.
\end{definition}

We can now formulate the Anick resolution and prove its exactness.

\begin{theorem}{Anick resolution}{}
    Let $\Kbb$ be a field. Let $A$ be a $\Kbb$-algebra augmented by $\eps$ with the section defined by $\eta(1_\Kbb) = 1_A$. Let $\freeAlgebra{X|R}$ be a presentation of $A$ such that $R$ is a minimal non-commutative Gröbner basis according to the monomial order $\prec$. Let $O(I(R)) := \freeAlgebra{X} \setminus \LM{I(R)}$ be the set of normal words. Let $V := \LM{R}$ be the set of leading monomials in $R$, called obstructions. For any $n \in \N$, let $C_n$ denote the set of $n$-chains on $V$ as defined in Definition~\ref{definition:chains-bottom-up}.

    The following is a free resolution of $\Kbb$ in the category of right $A$-modules:
    \[
        \cdots \rightarrow
        \Kbb C_{n+1} \otimes_{\Kbb} A \overset{d_{n+1}}{\rightarrow}
        \Kbb C_n \otimes_{\Kbb} A \rightarrow
        \cdots \rightarrow
        \Kbb C_2 \otimes_{\Kbb} A \overset{d_2}{\rightarrow}
        \Kbb C_1 \otimes_{\Kbb} A \overset{d_1}{\rightarrow}
        \Kbb C_0 \otimes_{\Kbb} A \overset{\eps}{\rightarrow}
        \Kbb \rightarrow 0,
    \]
    where for $n \geq 1$, the map of right $A$-modules $d_{n}$ satisfies:
    \begin{align*}
        \forall c^{(n)} \in C_n \vir \fP[d_{n}]{c^{(n)} \otimes 1_A} := \fB{c^{(n)}}^{n-1} \otimes \ol{\fB{c^{(n)}}_{n-1}} + \omega_{c^{(n)}},
    \end{align*}
    with either $\omega_{c^{(n)}}= 0$ or its high-term verifies $\LM{\omega_{c^{(n)}}} \prec c^{(n)}$.
\end{theorem}

\begin{proof}
    The proof is done by induction by constructing the differentials and contracting homotopy at the same time.

    Note that to prove exacteness at $E$ in $\cdots \rightarrow F \overset{\delta_1}{\rightarrow} E \overset{\delta_0}{\rightarrow} \cdots$, it suffices to prove that $\delta_0 \delta_1 = 0$ and that there exists a $\Kbb$-linear map $\fHV[\iota_0]{\ker(\delta_0)}{F}$ such that $\delta_1 \iota_0 = \id_{\ker( \delta_0)}$.
    
    Since $C_0 = \ensemble{1}$, we can identify $\Kbb C_0 \otimes_{\Kbb} A$ with $A$ for the initialisation, as a matter of simplifying notations. 
    
    Then, define $\fHV[d_1]{\Kbb C_1 \otimes_{\Kbb} A}{A}$ as the map of right $A$-modules with:
    \[
        \forall x \in C_1 = X \vir d_1(x \otimes 1_A) = \ol{x} - \eta\eps(\ol{x}).
    \]
    
    Firstly, it is evident that $\eps d_1 = 0$ since $\eps$ is $\Kbb$-linear map and $\eta$ is a section of it. The kernel of $\eps$ is spanned by the elements from $A$ of the form $\ol{s} - \eta\eps(\ol{s})$ where $s \in O(I(R))$. Indeed, every element of $a \in A$ is written $\eta\eps(a) + (a - \eta\eps(a))$ and we have the decomposition from the augmentation: $A = \Kbb 1_A \oplus \ker(\eps)$. Defining $i_0$ on those elements:
    \begin{align*}
        \forall s = x_1 \cdots x_\ell \in O(I(R)) \vir
            i_0(\ol{s} - \eta \eps(\ol{s})) &:= \sum_{j = 1}^{\ell} \fP[\eps]{\ol{x_{1} \cdots x_{j-1}}} \fP{x_j \otimes \ol{x_{j+1} \cdots x_{\ell}}}, \\
        \text{(with the convention that $x_1 \cdots x_{j-1} = 1$} & \text{ if $j \leq 1$ and $x_{j+1} \cdots x_\ell = 1$ if $j \geq \ell$)}
    \end{align*}
    and extending by $\Kbb$-linearity on $\ker(\eps)$ gives us a map satisfying $d_1 i_0 = \id_{\ker(\eps)}$. Indeed for $s = x_1 \cdots x_\ell \in O(I(R))$, we have:
    \begin{align*}
        d_1 i_0(\ol{s} - \eta \eps (\ol{s}))
            &= \fP[d_1]{\sum_{j = 1}^{\ell} \fP[\eps]{\ol{x_{1} \cdots x_{j-1}}} \fP{x_j \otimes \ol{x_{j+1} \cdots x_{\ell}}}} \\
            &= \sum_{j = 1}^{\ell} \fP[\eps]{\ol{x_{1} \cdots x_{j-1}}} \fP[d_1]{x_j \otimes \ol{x_{j+1} \cdots x_{\ell}}} &&\text{$d_1$ $\Kbb$-linear,}\\
            &= \sum_{j = 1}^{\ell} \fP[\eps]{\ol{x_{1} \cdots x_{j-1}}} \fP{\ol{x_j \cdots x_\ell} - \eta \eps(\ol{x_j}) \ol{x_{j+1} \cdots x_\ell}} && \text{definition of $d_1$,} \\
            &= \sum_{j = 1}^{\ell} \fP[\eps]{\ol{x_{1} \cdots x_{j-1}}} \ol{x_j \cdots x_\ell} - \eta \eps(\ol{x_{1} \cdots x_j}) \ol{x_{j+1} \cdots x_{\ell}} && \substack{\text{$\eta$ $\Kbb$-linear,} \\ \text{$\eps$ algebra morphism.}}
    \end{align*}

    Since $\eta(1_{\Kbb}) = 1_A$\footnote{This is a requirement, trivially verified when $\eta$ is considered as a morphism of unitary algebras and not simply $\Kbb$-linear. Otherwise, in that latter case, $\eta$, being a section of $\eps$, could satisfy $\eta(1_{\Kbb}) = 1_A + \omega$, where $\omega \in \ker(\eps)$}, then $\eta \eps(\ol{x_{1} \cdots x_j}) \ol{x_{j+1} \cdots x_{\ell}} = \eps(\ol{x_{1} \cdots x_j}) \ol{x_{j+1} \cdots x_{\ell}}$, therefore the right-most and left-most of two consecutive terms in the sum cancel out. Remain only the left- and right-most terms of the entire sum \ie:
    \[
        d_1 i_0(\ol{s} - \eta \eps (\ol{s})) = \ol{x_1 \cdots x_\ell} - \eta(\eps(\ol{x_1 \cdots x_\ell})) = \ol{s} - \eta \eps(\ol{s}).
    \]

    This proves exacteness of the sequence at $\Kbb C_0 \otimes_{\Kbb} A$.

    Now, suppose that, for $n \in \N$, the sequence:
    \[
        \cdots \rightarrow
        \Kbb C_{n+1} \otimes_{\Kbb} A \overunderset{d_{n+1}}{i_n}{\rightleftarrows}
        \Kbb C_{n} \otimes_{\Kbb} A \overunderset{d_{n}}{i_{n-1}}{\rightleftarrows}
        \Kbb C_{n-1} \otimes_{\Kbb} A \rightleftarrows
        \cdots \rightleftarrows
        \Kbb C_1 \otimes_{\Kbb} A \overunderset{d_1}{i_0}{\rightleftarrows}
        \Kbb C_0 \otimes_{\Kbb} A \overunderset{\eps}{\eta}{\rightleftarrows}
        \Kbb \rightarrow 0
    \]
    has been proven exact up to $\Kbb C_{n-1} \otimes_{\Kbb} A$ by defining the differentials $d_1$, ..., $d_{n}$ and the contracting homotopy maps $i_0$, ..., $i_{n-1}$ that verify for all $m \in \integers{1}{n}$: (letting $d_0 := \eps$)
    \begin{enum_roman}
        \item $d_{m-1} d_{m} = 0$. \label{proof:item:complex}
        \item $\forall c^{(m)} \in C_m \vir \fP[d_{m}]{c^{(m)}} = \fB{c^{(m)}}^{m-1} \otimes \ol{\fB{c^{(m)}}_{m-1}} + \omega_{c^{(m)}}$ \newline where either $\omega_{c^{(m)}} = 0$ or $\LM{\omega_{c^{(m)}}} \prec c^{(m)}$. \label{proof:item:differentials}
        \item $d_m i_{m-1} = \id_{\ker(d_{m-1})}$. \label{proof:item:acyclic}
        \item $\forall v \in \ker(d_{m - 1}) \setminus \ensemble{0} \vir \LM{i_{m-1}(v)} = \LM{v}\quad$ (equality in $\freeAlgebra{X}$ as in Definition~\ref{definition:LM}). \label{proof:item:contracting-homotopy}
    \end{enum_roman}

    We have proven Properties~\ref{proof:item:complex}, \ref{proof:item:differentials}, and \ref{proof:item:acyclic} are satisfied at initialisation. Moreover, it is quite evident from the definition of $i_0$ that Property~\ref{proof:item:contracting-homotopy} is verified.

    Let us define $d_{n+1}$ and $i_n$ such that they prove exactness at $\Kbb C_n \otimes_{\Kbb} A$.

    Define:
    \[
        \fHV[d_{n+1}]{\Kbb C_{n+1} \otimes_{\Kbb} A}{\Kbb C_n \otimes_{\Kbb} A}
        [c^{(n+1)} \otimes 1_A]
        [c^{(n)} \otimes \ol{t} - \fP[i_{n-1} d_n]{c^{(n)} \otimes \ol{t}}]
    \]
    where $c^{(n)} := \fB{c^{(n+1)}}^n$ and $t := \fB{c^{(n+1)}}_n$.

    By Property~\ref{proof:item:acyclic} for $m = n$, it is evident that $d_n d_{n+1} = 0$, proving Property~\ref{proof:item:complex} for $d_{n+1}$.

    Letting $\omega_{c^{(n+1)}} := -\fP[i_{n-1} d_n]{c^{(n)} \otimes \ol{t}}$ we obtain the desired expression for $d_{n+1}$. Indeed, by Property~\ref{proof:item:contracting-homotopy} for $m = n$, $\LM{\omega_{c^{(n+1)}}} = \LM{d_n(c^{(n)} \otimes \ol{t})}$. But by Property~\ref{proof:item:differentials}, $\LM{d_n(c^{(n)} \otimes \ol{t})}$ matches with $c^{(n-1)} \otimes \ol{r}$ where $c^{(n-1)} := \fB{c^{(n+1)}}^{n-1}$ and $r := \fB{c^{(n+1)}}_{n-1}$. But since there are no overlaps between the last obstructions of $c^{(n+1)}$ and of $c^{(n-1)}$, $r$ contains the last obstruction of $c^{(n+1)}$ and is therefore not normal. Reducing it to its normal form $\hat{r}$ implies that $c^{(n-1)}\hat{r}$ is smaller than $c^{(n+1)}$. Thus $\LM{\omega_{c^{(n+1)}}} \prec c^{(n+1)}$, proving Property~\ref{proof:item:differentials} for $d_{n+1}$.

    Let us define recursively the $\Kbb$-linear map $i_n$ on $\ker(d_{n})$. Let $v = \sum_{i} \lambda_i c^{(n)}_i \otimes \ol{s_i} \in \ker(d_n)$. Assume that $i_n$ has been defined for all $v' \in \ker(d_n)$ with $\LM{v'} \prec \LM{v}$ such that it satisfies Properties~\ref{proof:item:acyclic}~and~\ref{proof:item:contracting-homotopy} on those elements. Without loss of generality, assume that $\LM{v}$ coincides with $c^{(n)}_0 \otimes \ol{s_0}$ where $s_0$ is normal. Then, since $d_n(v) = 0$, it follows that $c^{(n-1)}_0 \otimes \ol{r_0} = - \omega_{v}$ where $\fP[d_n]{c^{(n)}_0 \otimes \ol{s_0}} = c^{(n-1)}_0 \otimes \ol{r_0} + \omega_{c^{(n)}_0}$ in such a way that $c^{(n-1)}_0 = \fB{c^{(n)}_0}^{n-1}$ and $r_0 = \fB{c^{(n)}_0}_{n-1} s_0$, as well as, $\omega_v = \frac{1}{\lambda_0} \fP{\omega_{c^{(n)}_0} + \fP[d_n]{\sum_{i \neq 0} \lambda_i c^{(n)}_i \otimes \ol{s_i}}}$ and thus by Property~\ref{proof:item:differentials}, $\LM{c^{(n-1)}_0 \otimes \ol{r_0}} \prec c^{(n)}_0 s_0$. This implies, since $c^{(n-1)}_0 r_0 = c^{(n)}_0 s_0$, that $r_0$ is not normal and thus contains an obstruction. Consider the obstruction in $r_0$ starting the furthest to the left. It will overlap with the last obstruction in $c^{(n)}_0$ since $s_0$ is normal. Therefore, we obtain an $(n+1)$-chain $c^{(n+1)}$ and $t$ a proper suffix of $s_0$ such that $c^{(n-1)}_0 r_0 = c^{(n)}_0 s_0 = c^{(n+1)}t$ in $\freeAlgebra{X}$. Define:
    \begin{equation}
        i_n(v) := \LC{v} c^{(n+1)} \otimes \ol{t} + \fP[i_n]{v - \LC{v} \fP[d_{n+1}]{c^{(n+1)} \otimes \ol{t}}}. \label{eq:contracting-homotopy}
    \end{equation}
    This works because $\LM{v - \LC{v} \fP[d_{n+1}]{c^{(n+1)} \otimes \ol{t}}} < \LM{v}$ since $d_{n+1}$ verifies the Property~\ref{proof:item:differentials} and thus cancellation on the leading term occurs. It follows that $i_n$ verifies Property~\ref{proof:item:contracting-homotopy} since we assumed $i_n$ satisfies Property~\ref{proof:item:contracting-homotopy} for elements with a smaller leading monomial.

    Finally, by recursive hypothesis, we have that Property~\ref{proof:item:acyclic} is verified on elements with a smaller leading monomial. Hence:
    \[
        d_{n+1} i_n (v) = \LC{v} \fP[d_{n+1}]{c^{(n+1)} \otimes \ol{t}} + v - \LC{v} \fP[d_{n+1}]{c^{(n+1)} \otimes \ol{t}} = v
    \]
    and thus, $d_{n+1}$ and $i_n$ will ultimately verify Property~\ref{proof:item:acyclic} on $\ker(d_n)$.
    
    This concludes the inductive proof.
\end{proof}

\begin{example}{}{}
    Let us consider, in the continuity of the examples throughout this paper, the algebra presented by $\freeAlgebra{X|R}$, where $X = \ensemble{x, y, z}$ and $R = \ensemble{xxyx, xxx - xx, yxz - yx}$ with the deglex monomial order induced by $x \succ y \succ z$ augmented with the evalutation of polynomials at zero. We have $V := \LM{R} = \ensemble{xxyx, xxx, yxz}$ and the graph of $n$-chains is given in Figure~\ref{fig:chains-graph}.

    We have, for all $\zeta \in X$ and $x_1 \cdots x_\ell \in O(I(R))$:
    \begin{align*}
        d_1(\zeta \otimes \ol{1}) &= \ol{\zeta} \\
        i_0(\ol{x_1 \cdots x_\ell}) &= x_1 \otimes \ol{x_2 \cdots x_\ell}
    \end{align*}

    Then:
    \begin{align*}
        d_2(xxx \otimes \ol{1})
            &= x \otimes \ol{xx} - i_0 d_1 (x \otimes \ol{xx})
                && \text{definition of $d_2$} \\
            &= x \otimes \ol{xx} - i_0 (\ol{xxx})
                && \text{definition of $d_1$} \\
            &= x \otimes \ol{xx} - i_0 (\ol{xx})
                && \text{reduction} \\
            &= x \otimes \ol{xx} - x \otimes \ol{x}
                && \text{definition of $i_0$}
    \end{align*}

    Similarly, we compute:
    \begin{align*}
        d_2(xxyx \otimes \ol{1})
            &= x \otimes \ol{xyx} \\
        d_2(yxz \otimes \ol{1})
            &= y \otimes \ol{xz} - y \otimes \ol{x}
    \end{align*}

    The $3$-chains are $\ensemble{xxyxxyx, xxyxxx, xxyxz, xxxyx, xxxx}$. Then:
    \begin{align*}
        d_3(xxxyx \otimes \ol{1})
            &= xxx \otimes \ol{yx} - i_1 d_2 (xxx \otimes \ol{yx})
                && \text{definition of $d_3$} \\
            &= xxx \otimes \ol{yx} - i_1 (x \otimes \ol{xxyx} - x \otimes \ol{xyx})
                && \text{definition of $d_2$} \\
            &= xxx \otimes \ol{yx} - i_1 (x \otimes 0 - x \otimes \ol{xyx})
                && \text{reduction} \\
            &= xxx \otimes \ol{yx} + xxyx \otimes \ol{1}
                && \text{definition of $i_1$ (see (\ref{eq:contracting-homotopy}))}
    \end{align*}
    In an analoguous manner, we compute:
    \begin{align*}
        d_3(xxyxxyx \otimes \ol{1}) &= xxyx \otimes \ol{xyx} \\
        d_3(xxyxxx \otimes \ol{1}) &= xxyx \otimes \ol{xx} - xxyx \otimes \ol{x} \\
        d_3(xxyxz \otimes \ol{1}) &= xxyz \otimes \ol{z} - xxyx \otimes \ol{1} \\
        d_3(xxxx \otimes \ol{1}) &= xxx \otimes \ol{x}
    \end{align*}

    The $4$-chains are:
    \begin{multline*}
        \{xxyxxyxxyx, xxyxxyxxx, xxyxxyxz, xxyxxxyx, \\
        xxyxxxx, xxxyxxyx, xxxyxxx, xxxyxz, xxxxxyx, xxxxxx\}
    \end{multline*}

    We thus have:
    \begin{align*}
        d_4 (xxxyxxx \otimes \ol{1})
            &= xxxyx \otimes \ol{xx} - i_2 d_3 (xxxyx \otimes \ol{xx})
                && \text{definition of $d_4$} \\
            &= xxxyx \otimes \ol{xx} - i_2 (xxx \otimes \ol{yxxx} + xxyx \otimes \ol{xx})
                && \text{definition of $d_3$} \\
            &= xxxyx \otimes \ol{xx} - i_2 (xxx \otimes \ol{yxx} + xxxyx \otimes \ol{xx})
                && \text{reduction} \\
            &= xxxyx \otimes \ol{xx} - xxxyx \otimes \ol{x} - i_2 (xxyx \otimes \ol{xx} - xxyx \otimes \ol{x})
                && \text{definition of $i_2$ (see (\ref{eq:contracting-homotopy}))} \\
            &= xxxyx \otimes \ol{xx} - xxxyx \otimes \ol{x} - xxyxxx \otimes \ol{1}
                && \text{definition of $i_2$ (see (\ref{eq:contracting-homotopy}))}
    \end{align*}
    We compute in the same way:
    \begin{align*}
        d_4 (xxyxxyxxyx \otimes \ol{1})
            &= xxyxxyx \otimes \ol{xyx} \\
        d_4 (xxyxxyxxx \otimes \ol{1})
            &= xxyxxyx \otimes \ol{xx} - xxyxxyx \otimes \ol{x} \\
        d_4 (xxyxxyxz \otimes \ol{1})
            &= xxyxxyx \otimes \ol{z} - xxyxxyx \otimes \ol{1} \\
        d_4 (xxyxxxyx \otimes \ol{1})
            &= xxyxxx \otimes \ol{yx} + xxyxxyx \otimes \ol{1} \\
        d_4 (xxyxxxx \otimes \ol{1})
            &= xxyxxx \otimes \ol{x} \\
        d_4 (xxxyxxyx \otimes \ol{1})
            &= xxxyx \otimes \ol{xyx} - xxyxxyx \otimes \ol{1} \\
        d_4 (xxxyxz \otimes \ol{1}) 
            &= xxxyx \otimes \ol{z} - xxxyx \otimes \ol{1} - xxyxz \otimes \ol{1} \\
        d_4 (xxxxxyx \otimes \ol{1})
            &= xxxx \otimes \ol{xyx} \\
        d_4 (xxxxxx \otimes \ol{1})
            &= xxxx \otimes \ol{xx} - xxxx \otimes \ol{x}
    \end{align*}

    We can compute in that fashion any differential, by computing all the previous ones that are needed.
\end{example}

	\section*{Acknowledgements}

	I would like to thank Cyrille \textsc{Chenavier} for his continuous guidance and support all along the process of writing this paper, providing many relevant remarks and insight on the subject. I would also like to thank Thomas \textsc{Cluzeau} for proofreading this note.

	\bibliography{references}{}
    \bibliographystyle{plain}
\end{document}